\documentclass[A4paper,11pt]{article}
 \usepackage{amsmath}

\usepackage[utf8]{inputenc}
\usepackage[a4paper, total={7.2in,9in}]{geometry}        
\usepackage[british,UKenglish,USenglish,english,american]{babel} 
\usepackage{pgf,tikz}
\usepackage{mathrsfs}

\usepackage{natbib}
\usepackage{fancybox}
\usepackage{url}  
\usepackage{mathtools}
\usepackage{amssymb}
\usepackage{amsthm}

\usepackage{enumerate}
\usepackage{multirow}
\usepackage{makeidx}
\makeindex
\usepackage{fancybox}
\usepackage{epsfig}
\usepackage{float}
\usepackage{hyperref}
\usepackage{array}

\usepackage{tabularx}

\usepackage{color}

\usepackage{graphicx} 
\usepackage{subcaption}

\usepackage{caption}
\usepackage{epstopdf}
\usepackage{slashbox}
\usepackage{srcltx}

\usepackage{pgf,tikz}
\usepackage{mathrsfs}
\usetikzlibrary{arrows}
\usepackage{stmaryrd}
\numberwithin{equation}{section}

\usepackage{calc}
\usepackage{ragged2e}

\usepackage{natbib}
\setcitestyle{numbers}

\usepackage{bm}
\usepackage[margin=2cm]{caption}

\usepackage{wrapfig}

\newcommand{\diff}{\mathop{}\!\mathrm{d}}

\newtheorem{thm}{Theorem}%
\newtheorem{lem}[thm]{Lemma}%
\newtheorem{dfn}[thm]{Definition}%
\newtheorem{rem}[thm]{Remark}%

\newcommand{\divergence}{\mathop{\rm div}\nolimits}

\newcommand{\fonction}[5]{\begin{array}{l|rcl}
		#1: & #2 & \longrightarrow & #3 \\
		& #4 & \longmapsto & #5 \end{array}}
\newcommand{\norme}[1]{\left\Vert #1\right\Vert}

\newcommand{\g}{{g}}
\newcommand{\w}{{w}}
\newcommand{\Om}{{\Omega}}

\newcommand{\tf}{{t_f}}
\newcommand{\al}{{\alpha}}
\newcommand{\grav}{{\textbf{g}}}
\newcommand{\Lam}{{\Lambda}}
\newcommand{\Gam}{\Gamma}
\newcommand{\GamN}{\Gam_\Neu}
\newcommand{\GamD}{\Gam_\Dir}
\newcommand{\Dir}{{Dir}}
\newcommand{\Neu}{{N}}
\newcommand{\V}{{\textbf{V}}}
\newcommand{\n}{{\textbf{n}}}
\newcommand{\M}{{\mathfrak{M}}}
\newcommand{\Nf}{{N}}
\newcommand{\dt}{{\delta t}}
\newcommand{\E}{{\mathcal{E}}}

\newcommand{\sig}{{\sigma}}
\newcommand{\T}{{\mathcal{T}}}
\newcommand{\D}{{\mathcal{D}}}
\newcommand{\DD}{{\mathfrak{D}}}
\newcommand{\R}{{\mathbb{R}}}

\newcommand{\ddKL}{\delta_{KL}}

\newcommand{\ddKeLe}{\delta_{K^*L^*}}

\newcommand{\ddAB}{\delta_{AB}}

\newcommand{\F}{{\mathcal{F}}}
\newcommand{\HH}{{\mathcal{H}}}
\newcommand{\gam}{{\gamma}}
\newcommand{\VV}{{\mathcal{V}}}
\newcommand{\VVD}{{\VV_{\Dir}}}
\newcommand{\VVDc}{{\VV_{\Dir}^c}}

\newcommand{\A}{{\mathcal{A}}}

\newcommand{\eps}{{\epsilon}}
\newcommand{\Mobi}{{M_{{\al_i}}}}
\newcommand{\sati}{{u_{{\al_i}}}}
\newcommand{\uisup}{{\overline{\sati}}}
\newcommand{\uiinf}{{\underline{\sati}}}

\newcommand{\pointa}{(P.a)~}
\newcommand{\pointb}{(P.b)~}
\newcommand{\pointc}{(P.c)~}
\newcommand{\pointd}{(P.d)~}
\newcommand{\pointe}{(P.e)~}

\newcommand{\CtePhiSup}{{\phi_1}}

\newcommand{\Lamsup}{{\overline{\Lam}}}
\newcommand{\Laminf}{{\underline{\Lam}}}
\newcommand{\rhosup}{{\rho_1}}
\newcommand{\rhoinf}{{\rho_0}}
\newcommand{\phisup}{{\phi_1}}
\newcommand{\phiinf}{{\phi_0}}

\newcommand{\CtePointd}{{C_\eta}}
\newcommand{\CtePointe}{{C}}

\newcommand{\Z}{{\mathcal{Z}}}

\newcommand{\pc}{{p_c}}

\usepackage{varioref}

\labelformat{theorem}{theorem~#1}
\labelformat{lemma}{lemma~#1}
\labelformat{proposition}{proposition~#1}
\labelformat{corollary}{corollary~#1}

\AtBeginDocument{

}

\title{\textbf{Existence of solutions to numerical schemes using regularization: application to two-phase flow in porous media schemes}}
\usepackage{authblk}

\author[1]{Thomas Crozon}

\affil[1]{{\small Nantes Université, Ecole Centrale de Nantes, LMJL, CNRS UMR 6629, 1 rue de la No\"e, 44321 Nantes, France.  thomas.crozon@ec-nantes.fr}}

\date{\today}
\begin{document}
\maketitle

The present document corresponds to the $4^{th}$ chapter of my thesis, the problem setting is not definitive, what matters most here are the mathematical results and the methodology of the existence proofs. 

\abstract{
In this work, we propose a framework and some tools for establishing the existence of solutions to numerical schemes in the case of the two-phase flow model. These schemes are sharing some key a priori mathematical properties. It applies to a large variety of continuous models. We propose the definition of a regularized scheme and show that if solutions exist to this regularization, then the existence of the initial one is ensured. This perturbation of the scheme facilitates the regularized existence. The main aim is to handle degenerate systems such as the two-phase Darcy flows in porous media. We illustrate the strength of our framework on two practical schemes, a finite volume one using the DDFV framework, and the other based on a Control Volume Finite Element (CVFE) method.
}

\section{Introducing the problem}\label{Section1Chap3}

In the two-phase Darcy flow, we have a coupled system of two degenerated parabolic equations \cite{BeaBac2012IntroductionTransportPorousMedia,Helmig1997MultiphaseFlow}. The study of such systems is of great interest for engineering applications such as nuclear waste management, enhanced oil and gas recovery, management of geothermal energy, etc. An accountable amount of schemes have been built in various cases for similar models, with a wide variety of numerical methods. For instance, one can refer to finite volume-type methods such as Two Point Flux Approximation \cite{SaaSaa2013StudyFullImplicitPetroleumFiniteVolCompressibleTwoPhase,EymHerMic2003MathematicalStudyPetroleumEngineeringScheme,BenKhaSaa2014ConvergenceFiniteVolGasWater}, it has also been done using finite elements \cite{ChaventJaffre1986MathematicalFEReservoirSimulation,Ohlberger1997ConvergenceMixedFEFiniteVolTwoPhasePorousMedia,HubHel2000NodeCenteredFiniteVolMultiphaseHeterogeneousPorousMedia}, and other kinds of methods, but we do not want to be exhaustive here. The degeneracy is the main obstacle to proving the existence of a solution to such schemes. Often, the proof lacks full clarification, since the problem is complex and tricky. The dependence and definition of the approximate quantities before proving the existence of solutions prevent us from using an a priori maximum principle. Moreover, the degeneracy forces the energy estimates to be based on the global pressure $p$ \cite{ChaventJaffre1986MathematicalFEReservoirSimulation,AntKazMon1989BoundaryValue} and on a capillary term $\xi$ (see \eqref{DefCapillaryTerm}) for continuous and discrete solutions. The existence problem due to the degeneracy of the mobilities has been treated in \cite{KhaSaa2010SolutionsCompressible} in the continuous case. The difficulty is circumvented by regularizing the problem, solving it, and then passing to the limit of the regularized solutions to establish that they are solutions to the initial system. We propose a standard regularization of the scheme's result to demonstrate the existence of solutions for a numerical method applied to degenerate problems.

In Section \ref{SectionContinuousModelChap3}, we introduce a quite general type of model to which we can apply our result, the only hypotheses are concerning:
\begin{itemize}
     \item The link between the "pressure" variables and the "saturation" ones, depending on the space. It is designed to take into account not only the capillary pressure and saturating relations but also the heterogeneity of the porous medium. 
     \item The saturations need to verify a maximum principle, keeping their "physical" meaning.
     \item Some "non-physical" variables $\zeta$, on which we have the energy estimates, with a strong link between them and the "pressures", to mimic the role of the global pressure. 
\end{itemize}
Our result can be especially applied to implicit schemes. Assuming we have a scheme holding a discrete maximum principle and discrete energy estimates on the non-physical quantity $\zeta$, we define the concept of regularization of such a scheme in Section \ref{SectionRegularizationOfSchemeChap3}. In Section \ref{Section2TheoremChap3}, we prove that if there exist solutions to a regularized scheme it implies the existence of solutions to the regular scheme.

In Section \ref{Section3ApplicationTwoPhaseFlowChap3}, we present the continuous compressible, immiscible, two-phase Darcy flow in porous media model, which verifies all the hypotheses of our framework (see Section \ref{SectionContinuousModelChap3}). Following, we apply our results to establish the existence of solutions for two schemes. First, we use this strategy on a Discrete Duality Finite Volume (DDFV) implicit scheme \cite{Crozon1}. We build a regularized scheme and prove the existence of the latter solutions, implying the result of the first one. Secondly, one shows that the result also works well in the case of a Control Volume Finite Element (CVFE) implicit scheme introduced in \cite{GhiQueSaa2020PositivityPreservingFiniteVolCompressibleTwoPhaseAni}, following the same path. Both proposed regularizations rely on the $\eps$-perturbation of the mobilities and adding a $\eta$-capillary pressure flow with positive transmissibility coefficients.

\subsection{Continuous model}\label{SectionContinuousModelChap3}

In this work, we are concerned with the existence of a solution to a numerical scheme, for a specific type of model, but we try to open up its scope as much as possible. Starting from a coupled system of degenerate parabolic equations, we assume the scheme has already been built. We have a maximum principle and some energy estimates on the solutions. Usually, the degeneracy of the mobilities is a major issue in proving their existence for the two-phase Darcy flow. We propose a toolbox accompanied by a strategy to prove the existence of a solution to our scheme quite easily and rigorously, at least in the two-phase flow context. We point out that the main result does not only focus on these equations and can include other variants. \\

We carry out the study in a domain $\Om$ of $\R^d$, open and bounded, with $d \ge 1$, up to a final time $\tf$ ($\tf$ can be equal to $+\infty$). One denotes $Q_\tf = \Om \times (0,\tf)$ . We assume that we have the unknowns $p = (p_{\al_i})_{i=1,...,l}$ (we call them "pressures"; they will be the principal unknowns) and $u = (\sati)_{i=1,...,l}$ (which can be called "saturation"), with an integer $l \ge 2$. One calls the $\al_i$ the "phases". For instance, in compressible two-phase flow, one has $\al_1 = \g$, $\al_2 = \w$. Our interest is in the approximation of the degenerate coupled parabolic equation system as follows  
\begin{equation}\label{ModelContinue1Chap3}
     \partial_t \gam_i - \divergence(\Mobi \Lam (\nabla p_{\al_i} + V^i)) = f^i,
\end{equation}
where $\gam_i$, $p_{\al_i}$, $\Mobi$, $\Lam$, $V^{\al_i}$, $f^i$ designate respectively the accumulation term, the principal unknown, the mobility, the permeability tensor, a potential (usually the gravitation potential), and a source term. The proposed result will apply to every type of system of equations as long as one has the three following assumptions \eqref{RelationPiUiChap3}-\eqref{RelationPiXiiChap3}. \\

First, one has a formal relation linking the principal unknowns with the saturations, reducing the number of principal unknowns to $l$, given by $G$, defined almost everywhere on $\Om$, continuous with respect to (w.r.t.) the $p_i$, such that 
\begin{equation}\label{RelationPiUiChap3}
 \begin{aligned}
     & G(x,p_{\al_1},...,p_{\al_i},,...,p_{\al_l})\\ 
     &=(G_1(x,p_{\al_1},...,p_{\al_l}),...,G_i(x,p_{\al_1},...,p_{\al_l}),...,G_l(x,p_{\al_1},...,p_{\al_l}))\\
     & = (u_{\al_1},...,\sati,...,u_{\al_l}).
 \end{aligned} 
 \end{equation}
It can be called the coupling relation.

  Secondly, we want the quantities $\sati$ to adhere to specific physical bounds of type 
 \begin{equation}\label{MaximumPrincipleChap3}
     \uiinf \le u_{\al_i} \le \uisup, \quad \text{for} \quad 1 \le i \le l. 
 \end{equation}
Those bounds are given by the nature of the quantities described. For instance, if the $\sati$ are concentrations, then $\uiinf =0$ and $\uisup = + \infty$; if we have saturations then, $\uiinf =0$ and $\uisup =1$. 
\begin{rem}
    In two-phase flow we have parabolic equations \eqref{ModelContinue1Chap3} where the mobilities $\Mobi(\sati)$ are positive, continuously increasing functions with respect to the $\al_i$-saturation. The degeneracy issue means that the mobilities are vanishing with the saturations: $\Mobi(\sati = \uiinf) = 0$. This degeneracy in multiple cases is an obstacle to proving the existence of solutions to an implicit numerical scheme. Then, we continuously extend the mobilities by their extremal values: $\Mobi(\sati) = \Mobi(\uiinf)$ for $\sati \le \uiinf$ and $\Mobi(\sati) = \Mobi(\uisup)$ for $\sati \ge \uisup$. This is necessary to treat two-phase flows.
\end{rem}

Last, we assume that we have some non-physical quantities $(\zeta_i)_{1,...,l}$ (not directly involved in the state laws), such that there exists a constant $C>0$, functions $g_i(.,.)$ continuous w.r.t. their second variable, two one-to-one integer functions $\sig_1$, $\sig_2 : \llbracket 1,l \rrbracket \rightarrow \llbracket 1,l \rrbracket$, with $\sig_1(i) \neq \sig_2(i)$ for all $i$, verifying for almost every $x$ in $\Om$
\begin{equation}\label{RelationPiXiiChap3}
\begin{aligned}
        & p_{\al_i} = \zeta_i + g_i(x, \sati), \quad  |g_i| \le C |p_{\al_{\sig_1(i)}} - p_{\al_{\sig_2(i)}}|, \\ 
        & \text{and Lipschitz continuous w.r.t $\sati$ independently of $x$}.
\end{aligned}
\end{equation}
In the two-phase flow context, the non-physical quantities will be the global pressure \cite{ChaventJaffre1986MathematicalFEReservoirSimulation}, and the functions $g_i$ will be the corrective pressures (see Section \ref{SectionTwoPhaseModelChap3}).

Here, we do not specify the boundary and initial conditions, but we keep them in mind. 

One highlights that the main hypotheses we need to keep to apply our results are \eqref{RelationPiUiChap3} and \eqref{RelationPiXiiChap3}. The hypothesis, \eqref{MaximumPrincipleChap3} justifies the maximum principle wanted in the scheme. The type of equations treated can differ from the one presented in \eqref{ModelContinue1Chap3} since one verifies the latter hypotheses and the scheme has the right properties.

\subsection{Regularized schemes}\label{SectionRegularizationOfSchemeChap3}

Let $\T$ be a space discretization of $\Om$ with $\sharp \T$ degrees of freedom (d.o.f.). We are looking for discrete vectors $p_{\al_1,...,\al_l,\T} = (p_{\al_1,\T},...,p_{\al_l,\T})$ with $p_{\al_i,\T} \in \R^\T = \R^{\sharp \T}$. First, we assume that a solution $(p_{\al_1,...,\al_l,\T}^{n-1})$ has already been computed for the $(n-1)$th time-step, we want to calculate an approximated solution for the next time-step. Our scheme is given, for every d.o.f. $A\in \T$ and for every phase $i$, by a solution of the following equations
\begin{equation}
    \F_{i,A}( p_{\al_1,...,\al_l,\T}, p_{\al_1,...,\al_l,\T}^{n-1}) = 0.
\end{equation}
Thus, we have $l \times \sharp \T$ discrete equations. In the following, we write \\
$\F_{i,A}^n( p_{\al_1,...,\al_l,\T}) := \F_{i,A}( p_{\al_1,...,\al_l,\T}, p_{\al_1,...,\al_l,\T}^{n-1})$, and designate the scheme by $\F$. We want to solve $\F^n(p_{\al_1,...,\al_l,\T}) = 0$. The scheme considered in the present work are implicit Euler scheme, other times discretizations enter the proposed framework such as the Crank-Nicolson one.

We assume the scheme has been designed to verify key elements. First, a maximum principle on the approximated saturations, which are still given by the coupling: 
\begin{equation}
    \uiinf \le u_{i,A}^k \le \uisup, \quad \forall i \in \llbracket 1,l \rrbracket, \quad \forall A \in \T, \quad \forall k \in \llbracket 0,n-1 \rrbracket.
\end{equation}
Secondly, one has a priori energy estimates on the discrete non-physical quantities $\zeta$, given by the relation \eqref{RelationPiXiiChap3}, where the constant $C_{\zeta}$ is depending on the mesh, time-step, the previous time-step solutions, physical data of the problem, such that for a given discrete norm $\norme{.}_\T$, one has
\begin{equation}
    \sum_{i=1}^l \norme{\zeta_{i,\T}}_\T^2 \le C_{\zeta}.
\end{equation}
Here we consider any discrete norm, in the following, we will consider discrete norms on discrete gradients.

For nonlinear and complex numerical schemes, it is not evident to prove the existence of a solution to the scheme in a quick classic fashion. For instance, in the two-phase Darcy flow, degeneracy is a big problem preventing us from using the classical fixed-point theorems. Then, we propose to introduce a regularization of the scheme such that it is easier to prove the existence of a solution to the regular one (see for instance Section \ref{Section3ApplicationTwoPhaseFlowChap3}), as follows in Section \ref{Section2TheoremChap3}, we show that it gives at least one solution to $\F$.
\begin{dfn}[regularized scheme]
    Let us fix two positive parameters $\eps$, $\eta \ge 0$. A regularized scheme of $\F$, is a $\eps, \eta$-parametrized scheme $\F^{\eps,\eta}$, such that
    \begin{itemize}
        \item \pointa $\F^{\eps,\eta}$ is continuous with respect to $\eps$ and $\eta$,
        \item \pointb $\F^{0,0} = \F$, 
        \item \pointc each solution of $\F^{0,\eta} $, verifies the maximum principle on the approximated saturations, 
        \item \pointd one has regularized energy estimates, with $\CtePointd$ depending on the mesh, the $(n-1)$-solution, the physical data, but independent of $\eps >0$ (it can depend on $\eta>0$), such that
        \begin{equation}\label{RegEnergyEstimates1Chap3}
        \sum_{i=1}^l \norme{\zeta_{i,\T}}_\T^2 + \eps \sum_{i=1}^l \norme{p_{\al_i,\T}}_\T^2 + \eta \sum_{i=1}^{l} \norme{p_{\al_{\sig_1(i)},\T} - p_{\al_{\sig_2(i)},\T}}_\T^2\le \CtePointd.
        \end{equation}
        \item \pointe for $\eps = 0$ and $\eta >0$ regularized energy estimates stands with $\CtePointe$ depending on the mesh, the $(n-1)$-solution, the physical data, but independent of $\eta >0$, such that
        \begin{equation*}\label{RegEnergyEstimates2Chap3}
        \sum_{i=1}^l \norme{\zeta_{i,\T}}_\T^2 \le \CtePointe.
        \end{equation*}
    \end{itemize}
\end{dfn}
\begin{rem}
    The regularized energy estimate \eqref{RegEnergyEstimates1Chap3} is equivalent to
        \begin{equation}\label{RegEnergyEstimates1_2Chap3}
        \sum_{i=1}^l \norme{\zeta_{i,\T}}_\T^2 + \eps \sum_{i=1}^l \norme{p_{\al_i,\T}}_\T^2 + \eta \sum_{i=1}^{l-1} \norme{p_{\al_i,\T} - p_{\al_{i+1},\T}}_\T^2\le C,
    \end{equation}
    since $\sig_1$ and $\sig_2$ are two integer functions bijectives and never equal. The constant $C$ can vary. We will prefer these energy estimates after that since they are easier to handle.
\end{rem}
\pointb makes clear that $\F^{\eps,\eta}$ is a regularization of $\F$. The fact that we have the $\eps$,$\eta$-components in \pointd is often a key to proving the existence of regularized solutions. 

\subsection{Existence theorem for $\F$}\label{Section2TheoremChap3}

This part aims to prove that since we are able to show the existence of solutions for every $\eps >0$, $\eta > 0$, no matter the way we prove this existence, it follows there remains not less than one solution to  $\F=\F^{0,0}$. We will illustrate that it is easier to show the existence of a regularized scheme in Section \ref{Section3ApplicationTwoPhaseFlowChap3}.

\begin{thm}
    Assuming we have a solution to the regularized scheme $\F^{\eps,\eta}$ for every $\eps >0$, $\eta > 0$ (we will call them regularized solutions, and write it $p_{\al_1,...,\al_l,\T}^{\eps,\eta}$). Then, there exists a solution to the scheme $\F$.
\end{thm}

\begin{proof}[Proof]
One fixes $\eta >0$. Since the regularized energy estimates \eqref{RegEnergyEstimates2Chap3} given by \pointd are fulfilled, one infers the uniform bounds w.r.t $\eps >0$:
\begin{equation*}\label{UniformBoundEnergy}
        \left\{
        \begin{aligned}
            &( \zeta_{i,\T}^{\eps, \eta})_{\eps>0}  & \text{is uniformly bounded for every $i$ in }& \llbracket 1,l \rrbracket,& \\
            &( \sqrt{\eps} p_{\al_i,\T}^{\eps, \eta})_{\eps>0} & \text{is uniformly bounded for every $i$ in }& \llbracket 1,l \rrbracket,& \\
            &( \sqrt{\eta} (p_{\al_i,\T}^{\eps, \eta} - p_{\al_{i+1},\T}^{\eps, \eta}))_{\eps>0} \quad & \text{is uniformly bounded for every $i$ in }& \llbracket 1,l-1 \rrbracket.& \\
        \end{aligned}
        \right.
    \end{equation*}
We do not specify the norm because of the norm equivalence in finite dimension.
Thanks to relation \eqref{RelationPiXiiChap3}, because of Lipschitz assumption on $g_i$, the $\eta$ fixed and the uniform bound on $\sqrt{\eta} (p_{\al_i,\T}^{\eps, \eta} - p_{\al_{i+1},\T}^{\eps, \eta})$, we obtain that 
\begin{equation*}
    (  p_{\al_i,\T}^{\eps, \eta})_{\eps>0} \quad \text{is uniformly bounded for every $i$ in $\llbracket 1,l \rrbracket$}.
\end{equation*}
So, taking a sequence $(\eps_n)$ of strictly positive real numbers, converging to $0$ as $n$ goes to $\infty$, we have the sequence $(p_{\al_1,...,\al_l,\T}^{\eps_n, \eta})_n$, uniformly bounded in $(\R^\T)^l$. Because the dimension is finite, it is possible to extract a converging subsequence towards a limit, reading $p_{\al_1,...,\al_l,\T}^{\eta}$, which is solution of $\F^{0,\eta}$ because of \pointa and the continuity of $\F$. 

The third property \pointc, implies $p_{\al_1,...,\al_l,\T}^{\eta}$ verifies the maximum principle \eqref{MaximumPrincipleChap3}, for every $\eta >0$. Using \pointe, we have the uniform bounds but this time w.r.t. $\eta>0$ :
\begin{equation*}\label{UniformBoundEnergy2} 
        \begin{aligned}
            &( \zeta_{i,\T}^{ \eta})_{\eta>0}  & \text{is uniformly bounded for every $i$ in }& \llbracket 1,l \rrbracket.& \\
        \end{aligned}
\end{equation*}

The kept relations at the discrete level \eqref{RelationPiXiiChap3} and \eqref{RelationPiUiChap3}, combined with the maximum principle, Lipschitz-continuity of $g_i$ and the $\eta$-energy estimates, gives us that for all $1 \le i \le l$ and for all $A \in \T$:
\begin{equation*}
    |p_{\al_i,A}^{\eta}| \le |\zeta_{i,A}^{\eta}| + \norme{g_i}_{\infty, [ \uiinf, \uisup]} .
\end{equation*}
It implies 
\begin{equation*}
 ( p_{\al_i,\T}^{ \eta})_{\eta>0}  \quad \text{is uniformly bounded w.r.t. $\eta$, for every $i$ in } \llbracket 1,l \rrbracket. 
\end{equation*}
Then, taking $(\eta_n)$ a strictly positive sequence converging to $0$, one can extract a convergent subsequence of $(p_{\al_1,...,\al_l,\T}^{\eta_n})$, written $p_{\al_1,...,\al_l,\T}$. Because of the continuity \pointb, and the first property of the regularized scheme \pointa, $p_{\al_1,...,\al_l,\T}$ is a solution of the numerical scheme $\F$.     
\end{proof}
In this proof, we see that the hypotheses on the model \eqref{RelationPiUiChap3}-\eqref{RelationPiXiiChap3}, and those on the scheme and regularized scheme, are made to enable passing to the limit. Moreover, the $\epsilon$-part of the regularization energy estimates \pointd will be a crucial element in demonstrating the existence of regularized solutions more easily.

\section{Application to two schemes}\label{Section3ApplicationTwoPhaseFlowChap3}

Both of our applications are built on the same two-phase Darcy flow in porous media model that we will display in the following subsection \ref{SectionTwoPhaseModelChap3}. There are a lot of schemes on those models, but we will focus on two kinds of complex numerical schemes to demonstrate how our approach handles these problems. In a first time, we consider a pure finite volume scheme. In a second time, the proposed strategy is applied to a combined finite volume finite element scheme.

\subsection{The model}\label{SectionTwoPhaseModelChap3}

We are interested in the compressible two-phase flow in porous media problem \cite{BeaBac2012IntroductionTransportPorousMedia,GhiQueSaa2020PositivityPreservingFiniteVolCompressibleTwoPhaseAni,ChaventJaffre1986MathematicalFEReservoirSimulation}. In this model we have $l =2$ phases: a gazeous phase $\al_1 = \g$ and a wetting one $\al_2 = \w$. The capillary pressure $\pc : \R \rightarrow \R$  is an increasing homeomorphism, piecewise $C^1$ on $\R$, with bounded derivatives. Moreover, it verifies $\pc(0) = 0$. This work can be easily adapted to the heterogeneous medium case where the capillary pressure differs from one rock type to another. Then, we have the link between phase pressures and saturations (see \eqref{RelationPiUiChap3}) given by
\begin{equation}\label{RelationPiUiModel}
    G(x,p_{\g},p_{\w})= (s_\g,s_\w) = \left(\pc^{-1}(p_\g - p_\w), 1 - \pc^{-1}(p_\g - p_\w) \right). 
\end{equation}
The system is composed of two parabolic degenerate equations that are derived from the mass conservation for each phase. It reads 
\begin{equation} \label{BeginModel}
	\left\{
	\begin{array}{l}
		\phi \partial_t (\rho_{\al} s_{\al}) + \divergence (\rho_{\al} \V_{\al} ) + \rho_{\al} q^{\al} = 0 \quad \hbox{in} \quad Q_{\tf }, \quad \forall \al \in \{ \g, \w\}, \\
        \V_{\al} = - \dfrac{K_{r\al}}{\mu_{\al}} \Lam (\nabla p_{\al} - \rho_{\al} \grav), \quad \forall \al \in \{ \g, \w\}.
	\end{array}
	\right.
\end{equation}
In the first equation, $\phi$ refers to the porosity of the medium, and for each $\al$-phase $\rho_{\al}(p_{\al})$, $s_{\al}$, $q^{\al}$, $\V_{\al}$ respectively stand for the density depending only on the phase pressure $p_\al$, the saturation, the source term, the velocity. Each phase velocity is given in the second equation by the diphasic Darcy-Muscat law, where the mobility $M_\al(s_\al) =  K_{r\al} / \mu_{\al}$ is the relative permeability over the dynamic viscosity, $\Lam$ the permeability tensor of the medium and $\grav$ the gravitational acceleration. 

The saturations have a physical range (see \eqref{MaximumPrincipleChap3}) given by 
\begin{equation}\label{MaximumPrincipleModel}
    0 \le s^{\al} \le 1, \quad \text{for} \quad \al \in \{\g, \w \}. 
\end{equation}

The mobilities $M_\al$ are continuously increasing with respect to the saturation, positive, and degenerated. We consider their extensions $M_\al(u) =0$ for $u \le 0$ and $M_\al(u) =M_\al(1)$ for $u \ge 1$. We will consider in the following sections the $\eps$-regularized mobilities
\begin{equation}\label{EpsilonMobilities}
    M_\al^\eps(s) = \eps + M_\al(s), \quad \forall s \in \R.
\end{equation}

Here are the classical hypotheses on the main data. \begin{itemize}
    \item The porosity is bounded almost everywhere on $\Om$ by two strictly positive constant $\phiinf$, $\phisup>0$. It writes $ \phi \in L^{\infty}(\Om)$ with $\phiinf \le \phi(x) \le \phisup$ for a.e. $x$ in $\Om$.
    \item The permeability tensor is a symmetric positive-definite matrix, which is essentially bounded. Moreover, it is uniformly elliptic i.e. there exist constants $\Laminf$ and $\Lamsup$ such that 
	\begin{equation} \label{HypoPermea}\Laminf |v|^2 \le \Lam(x) v\cdot v \le \Lamsup |v|^2 \quad \text{ for all} \quad v \in \R^d  \quad \text{and a.e.} \quad x \in \Om . \end{equation}
    \item The density $\rho_{\al} \in C^1(\R,\R)$ is increasing (with the pressure) and uniformly bounded 
    \begin{equation*} 0 < \rhoinf \le \rho_{\alpha}(p_{\alpha}) \le \rhosup ,\end{equation*} for some positive constants $\rhoinf$, $\rhosup$.
\end{itemize}
The system is closed by compatible initial conditions on the pressures, with Neumann and Dirichlet boundary conditions:
\begin{equation*} \label{EndModel}
	\left\{
	\begin{array}{cc}
		\rho_{\g} \V_{\g} \cdot \n = \rho_{\w} \V_{\w} \cdot \n =0   & \text{  on  }   \Gam_\Neu \times (0,\tf ) , \\
		p_{\g} = p_{\g}^\Dir \text{  and  }  p_{\w} = p_{\w}^\Dir  & \text{  on  }   \Gam_\Dir \times (0,\tf),   \\
		p_{\g}(.,0) = p_{\g}^{ini} ~\text{  and  }  p_{\w}(.,0) = p_{\w}^{ini} & \text{ in  }    \Om ,
	\end{array}
	\right.
\end{equation*}
where $\{ \Gam_\Dir, \Gam_\Neu \}$ is a partition of the border $\partial \Om = \Gam_\Dir \cup \Gam_\Neu$ with $|\Gam_\Dir|>0$. We write $\n$ the outward unit normal of the Neumann border. 
We carry out our study in two dimensions. For the sake of simplicity, we assume we are in the case of a horizontal domain such that we can neglect the gravitational terms; moreover, there will be no source terms and we will set uniform Dirichlet boundary conditions $p_\g^\Dir=p_\w^\Dir=0$. Adding these ingredients is not a problem, it only complicates the writing of the formulas without hiding conceptual difficulties.

The concept of global pressure has been introduced in \cite{ChaventJaffre1986MathematicalFEReservoirSimulation}. This non-physical pressure is very useful in the analysis of the scheme to handle the degeneracy issue. We define the total mobility by $M(s_\g) = M_\w(1-s_\g) + M_\g(s_\g) \ge m_0 > 0$. This artificial pressure $p$, defined using corrective pressures $\hat{p}_\g $, $ \hat{p}_\w$ in the following way  
\begin{equation*} \label{GlobalPressureChap3}
\begin{aligned}
    & p_\g = p + \hat{p}_\g(s_\g), \quad p_\w = p - \hat{p}_\w(s_\g) \\
    & \text{where} \quad 
	\left\{
	\begin{array}{l}
		\displaystyle	\hat{p}_\g(s_\g) = \int_0^{s_\g} \dfrac{M_\w(1-u)}{M(u)} p'_c(u) \diff u \\
		\displaystyle	\hat{p}_\w(s_\g) = \int_0^{s_\g} \dfrac{M_\g(u)}{M(u)} p'_c(u) \diff u \\
	\end{array}
	\right. .
\end{aligned}
\end{equation*}
We see that we verify the hypothesis of \eqref{RelationPiXiiChap3}, with the same non-physical quantity for $p_\g$ and $p_\w$: the global pressure $p$. The corrective pressures are Lipschitz-continuous w.r.t. their respective saturations because the capillary pressure has a bounded derivative. Moreover, one writes 
\begin{equation}
\begin{aligned}
    & \left|\hat{p}_\g(s_\g)\right| = \left| \int_0^{s_\g} \dfrac{M_\w(1-u)}{M(u)} p'_c(u)  \diff u \right| \le \left| \int_0^{s_\g} p'_c(u)  \diff u \right| \le \left| p_\g - p_\w \right| \\ & \text{and similarly} \quad \left|\hat{p}_\w(s_\g)\right|  \le \left| p_\g - p_\w \right|.
\end{aligned}
\end{equation}

We will also use the feature of the function $\xi$ in energy estimates, of great help to prove the convergence, defined by
\begin{equation}\label{DefCapillaryTerm} 
\displaystyle	   \xi(s_\g)= \int_0^{s_\g} \dfrac{\sqrt{M_\w(1-u)M_\g(u)}}{M(u)} p'_c(u) \diff u .	
\end{equation}
The following nonlinear functions \cite{KhaSaa2010SolutionsCompressible} are of great use to show the energy estimates
\begin{equation}\label{NolinearFunctions}
    \displaystyle g_{\al}(p_{\al}) =  \int_0^{p_{\al}} \dfrac{1}{\rho_{\al}(b)}\diff b \quad \text{and} \quad  \HH_{\al}(p_{\al}) = {\rho_{\al}}(p_{\al}) g_{\al}(p_{\al}) -p_{\al}.
\end{equation}

\subsection{Positivity-Preserving DDFV scheme for compressible two-phase flow}\label{BPDDFVApplication}

First, we introduce the Discrete Duality Finite Volumes (DDFV) setting, and then we present the PP-DDFV scheme \cite{Crozon1}. We build a regularized scheme and prove the existence of the regularized solutions.

\subsubsection{DDFV settings}

We study a scheme in a $2D$ domain $\Om \subset \R^2$. First, we describe briefly the three meshes used in the DDFV-method \cite{DomOmn2005FiniteVolLaplaceEqu2DgridsDDFV, AndBoyHub2007DDFVLerayLionsElliptic2D,Krell2010These} to set up the useful notations. The description of the different types of meshes is inspired from \cite{IbrQueSaa2020positiveDDFVDegenerateParabolicChemotaxis,KreMoa2023StructurePreservingDriftDiffusionGeneralMeshesDDFVvsHFV}. \\

\textbf{The primal mesh:} \\ 
The primal interior mesh $\M$, is a collection of open disjoint polygons called primal cells, usually written $K$  covering $\Om$ (i.e. $\cup_{K \in \M} \overline{K} = \overline{\Om}$). We let $\partial \M$ be the set of boundary edges, which can be seen as degenerate cells. $\overline{\M}$ is then defined as the reunion of $\M$ and $\partial \M$. For each cell $K \in \overline{\M}$, we fix a point $x_K$ called its center, and we set $X_{int} = \{ x_K , K \in \M \}$, $X_{ext} = \{ x_K , K \in \partial \M \}$ and we write $X = X_{int} \cup X_{ext}$. The vertices of $\overline{\M}$ are split between those in the interior and the ones on frontier $X^* = X_{int}^* \cup X_{ext}^*$. For two neighboring primal cells $K$ and $L$, we assume $\partial K \cap \partial L = K|L$ is a segment, corresponding to an internal edge of the mesh if both cells are interior (in $\E_{int}$) or an exterior edge if one of the cells is in $\partial \M$ then the exterior cell can be confounded with the edge. We set $\E = \E_{int} \cup \partial \M$, the set of all the edges of the primal mesh. For a cell $K$, one denotes $\E_K$ the set of its edges, we distinguish the interior interfaces $\E_{K,int}$, from the exterior ones $\E_{K,ext}$. One sets the outward unit normal to $\sig \in \E_K$ as $\n_{\sig,K}$. 
\\

\textbf{The dual mesh:} \\ 
The dual control volumes are centered on the elements of $X^*$, written $x_{K^*}$ for a dual cell $K^*$. A cell $K^*$ is built by straightly joining, in the circular sense, the centers of the primal cells sharing the underlined vertex. When $x_{K^*}$ is in $X^*_{ext}$, we connect $x_{K^*}$ to the two midpoints of the two exterior primal edges sharing $x_{K^*}$ as a vertex. The dual edges $\sig^* \in \E^*$ are the segments linking the centers of the adjoining primal cells and when the two cells are at the border, we take the two segments connecting each center with the vertex.  We denotes $\M^*$, $\partial \M^*$ respectively the dual volumes constructed from $X^*_{int}$, $X^*_{ext}$. Then, one write the dual mesh $\overline{ \M^*} = \M^* \cup \partial \M^*$. Like in the primal case, one defines $\n_{\sig^*,K^*}$ the outward unit normal to $\sig^* \in \E_{K^*}$, for all $K^* \in \overline{ \M^*}$. 
\\

\textbf{The diamond mesh:} \\ 
For $\sig =K|L$ ($K$ or $L$ can be in $\partial \M$) with vertices $x_{K^*}$ and $x_{L^*}$, we define the quadrilateral diamond $\D_{\sig, \sig^*}$ $(\sig = K|L, \sig^*=K^*|L^*)$, whose vertices are $x_K$, $x_{K^*}$, $x_L$ and $x_{L^*}$. The diamond is built by connecting precedent points in the same order. When $\sig \in \partial \M$, the diamond degenerates into a triangle. One denotes $\D$ a diamond cell and $\DD$ the diamond mesh. Notice that $\D$ does not necessarily have a convex shape, it depends  on the location of the vertices of $\sig$. One defines $\al_{\D}$ the angle between the interfaces i.e. the angle between $(x_K,x_L)$ and $(x_{K^*},x_{L^*})$. \\

\textbf{Boundary conditions:}
We make the assumption that $\GamD \cap \GamN$ is a set of vertices of the primal mesh, then, the centers of the primal boundary cells are exclusively in $\GamD$ or $\GamN$. Thus, we divide $\partial \M$ between $\partial \M_\Dir = \left\{ K \in \partial \M, x_K \in \GamD \right\}$ and $\partial \M_\Neu = \left\{ K \in \partial \M, x_K \in \GamN \backslash \GamD \right\}$. As for the primal mesh, one sets \\ $\partial \M_\Dir^* = \left\{ K^* \in \partial \M^*, x_{K^*} \in \GamD \right\}$ and $\partial \M_\Neu^* = \left\{ K \in \partial \M^*, x_{K^*} \in \GamN \backslash \GamD \right\}$. We also define $\overline{\M}_\Neu = \M \cup \partial \M_\Neu$ and $\overline{\M^*}_\Neu = \M^* \cup \partial \M^*_\Neu$, since we will have homogeneous Dirichlet boundary conditions on pressures in our scheme and we will look after those such discrete solutions.  \\

We denote $\T= \left( \overline{\M}, \overline{\M^*} \right)$ the DDFV mesh. For any $A$ in $\M$, $\overline{\M^*}$ or $\DD$, $m_A$, $d_A$ stands respectively for the $d$-Lebesgue measure and the diameter of the cell. Similarly for $\Tilde{\sig} \in \E, \E^*$, $m_{\Tilde{\sig}}$ is the $(d-1)$-Lebesgue measure, or its length. The diamond measure can be computed via : $ \displaystyle m_{\D} = \dfrac{1}{2} m_{\sig} m_{\sig^*} \sin(\al_{\D})$.

\subsubsection{Discrete operators and functions}

One defines the discrete spaces $\R^{\T_\Dir}$, $\R^\T$ and $\R^\T_\Dir$. We have 
\begin{equation*} 
\begin{aligned}
    & u_{\T_\Dir} = \Big((u_K)_{K \in \overline{\M}_\Neu} ,(u_{K^*})_{K^* \in \overline{\M^*}_\Neu} \Big) \in \R^{\T_\Dir}, \\
    & \quad u_{\T} = \Big((u_K)_{K \in \overline{\M}} ,(u_{K^*})_{K^* \in \overline{\M^*}} \Big) \in \R^{\T},
\end{aligned} 
\end{equation*} 
and $\R^\T_\Dir$ is composed of the elements $u_\T$ of $\R^\T$ such that $u_A = 0$ for every $A \in \partial \M_\Dir \cup \partial \M^*_\Dir$. We have a linear injection $\R^{\T_\Dir} \rightarrow \R^\T_\Dir \subset \R^\T$. In the following, $u_\T$ will refer to an element of $\R^{\T_\Dir}$ or $\R^\T_\Dir$, depending on the context. If $f: \R \rightarrow \R$ is a nonlinear function, we denote by $f(u_{\T})$ the vector: 
\begin{equation*} f(u_{\T}) = \Big((  f(u_K))_{K \in \overline{\M}} ,(f(u_{K^*}))_{K^* \in\overline{\M^*}} \Big) \in \R^\T.\end{equation*} 
For any $g$ in $\R^\T$, one writes  
\begin{equation} 
\ddAB g = g_B - g_A, \quad \forall A,B \in \overline{\M} \cup \overline{\M^*}. 
\end{equation}
Following, $(\R^2)^{\DD}$ stands for the set of vector fields, composed of piecewise constants on diamonds, of the form: $\zeta_{\DD} = (\zeta_{\D})_{ \D \in \DD}$. In the DDFV approach, the discrete gradient operator is a linear mapping from $\R^{\T}$ (or $\R^{\R_\Dir}$) to $(\R^2)^{\DD}$, its purpose is to mimic a gradient \cite{DomOmn2005FiniteVolLaplaceEqu2DgridsDDFV,AndBoyHub2007DDFVLerayLionsElliptic2D}. It is defined for every $u_{\T} \in \R^{\T}$ by: 
\begin{equation*}
\begin{aligned}
    & \nabla^{\DD}u_{\T} = \sum_{ \D \in \DD} \nabla^{\D}u_{\T} \mathbf{1}_{\D}, \\
    & \text{with} \quad  \nabla^{\D}u_{\T} = \dfrac{1}{\sin (\al_{\D})} \left( \dfrac{\ddKL u}{m_{\sig^*}} \n_{\sig,K} +  \dfrac{\ddKeLe u}{m_{\sig}}\n_{\sig^*,K^*} \right) , \quad \forall \D \in \DD,
\end{aligned}
\end{equation*}
where $\mathbf{1}_{\D}$ is the characteristic function of $\D$. If $\D$ has an edge on the boundary, the value $u_L$ is assumed to be known and imposed from the boundary conditions (Dirichlet or Neumann). 

The permeability or stiffness tensor is approximated on the diamond using its mean value on $\D$  
\begin{equation} \Lam_{\D} = \dfrac{1}{m_{\D}}\int_{\D} \Lam(x) \diff x .\end{equation}
Now, we can give the transmissibility coefficients :
\begin{equation}\label{Transmissibilities}
\begin{aligned}
    	 \tau_{KL} = \dfrac{m_{\sig}}{  m_{\sig^*}} \dfrac{ \langle \Lam_{\D} \n_{KL}, \n_{KL} \rangle}{\sin( \al_{\D})} > 0, & \quad \tau_{K^*L^*} = \dfrac{m_{\sig^*}}{  m_{\sig}}\dfrac{ \langle \Lam_{\D} \n_{K^*L^*}, \n_{K^*L^*} \rangle}{\sin( \al_{\D})} > 0, \\
         & \eta_{\D} = \dfrac{ \langle \Lam_{\D} \n_{KL}, \n_{K^*L^*} \rangle}{\sin( \al_{\D})}. \\
\end{aligned}
\end{equation}
We equip the finite-dimensional space $\R^{\T}$ (or $\R^{\T_\Dir}$) with the $L^p$-semi-norm $|.|_{p, \T}$ (which is a norm on $\R^{\T_\Dir}$), as follow, for $u_{\T} \in \R^{\T}$
\begin{equation*} 
|u_{\T}|_{p, \T} = \left( \dfrac{1}{2} \sum_{K \in \M} m_K |u_K|^p + \dfrac{1}{2} \sum_{ K^* \in \overline{ \M^*}} m_{K^*} |u_{K^*}|^p \right)^{\dfrac{1}{p}} \quad \mbox{with } 1 \le p < +\infty.
\end{equation*}
One can define two norms on $\R^{\T_\Dir}$, using the DDFV gradient, such as
\begin{equation*} 
\begin{aligned}
    \norme{ u_{\T} }_{\T,\DD}^2 & = \sum_{ \D \in \DD} m_{\D} \norme{\nabla^{\D} u_{\T}}^2, \\
    \norme{u_\T}_{\T,\tau}^2 & = \sum_{\D \in \DD} \tau_{KL} (\ddKL u)^2 + \tau_{K^*L^*} (\ddKeLe u)^2 .\\
\end{aligned}
\end{equation*} 
We are able to show, since we have the permeability verifying \eqref{HypoPermea}, that $\norme{.}_{\T,\DD}$ and $\norme{.}_{\T,\tau}$ are two equivalent norms. Moreover, there exists $C>0$ depending only on the mesh regularity, for all $u_\T$ in $\R^{\T_\Dir}$ such that 
\begin{equation} \label{IneqNorms}
    |u_\T|_{1,\T} \le C \norme{u_\T}_{\T,\tau}. 
\end{equation}

\subsubsection{Presentation of the scheme}

We introduce the implicit PP-DDFV finite volume method (see \cite{Crozon1,CrozonQueSaa2023FVCA10DDFVTwoPhaseIncompressible}) \eqref{DebutSchemeDDFV}-\eqref{rhoKL}. We split the time interval into subintervals $[t^n,t^{n+1}[$ such that $0=t_0 < t^1< ...< t^\Nf= \tf$. One denotes $\dt = t^{n+1}-t^n$, it could be taken uniform, but it does not impact the main  result. We assume that  $p_{\T}^n$ in $\R^{\T_\Dir}$ is verifying \eqref{MaximumPrincipleModel}. We take the notation $|\g| = 0$, $|\w| = 1$. For simplicity, we omit the implicit time $n+1$ superscript. We are looking for $p_{\T}$ in $\R^{\T_\Dir}$ solution to the regularized scheme, given by $\F^{\eps,\eta}$ as follow, with $\eps \ge 0$, $\eta \ge 0$
\begin{itemize}
    \item[-] For $K \in \M$,
    \begin{equation}\label{DebutSchemeDDFV}
\begin{aligned}
    & \F_{\al,K}^{\eps,\eta} (p_{\T}, p_{\T}^{n}) 
    \\
    &= m_K \phi_K \left( \rho_{\al}(p_{\al,K}) \Z(s_{\al,K}) - \rho_{\al}(p_{\al,K}^{n}) s_{\al,K}^{n} \right) \\
    & -  \dt  \sum_{\sig = K|L \in \E_K} \rho_{\al, KL} V_{KL}^{\al,\eps} - \dt \eta (-1)^{|\al|}\sum_{\sig = K|L \in \E_K} \rho_{\al, KL} p_{c,KL}.
\end{aligned}
\end{equation}
\item[-] For $K^* \in  \M^* \setminus \M^*_\Dir$,
\begin{equation}\label{EqDual}
	\begin{aligned}
		& \F_{\al,K^*}^{\eps,\eta} (p_{\T}, p_{\T}^{n})\\
        & =  m_{K^*} \phi_{K^*} \left( \rho_{\al}(p_{\al,K^*}) \Z(s_{\al,K^*}) - \rho_{\al}(p_{\al,K^*}^{n}) s_{\al,K^*}^{n}   \right)  \\ 
        & - \dt  \sum_{\sig^* = K^*|L^*\in \E_{K^*}} \rho_{\al, K^*L^*}V_{K^*L^*}^{\al, \eps} - \dt \eta (-1)^{|\al|}\sum_{\sig = K^*|L^* \in \E_K^*} \rho_{\al, K^*L^*} p_{c,K^*L^*}.
	\end{aligned}
\end{equation}
\item[-] For $K \in \partial \M_\Neu$ ($K = K|L$ with $L\in \M$),
\begin{equation}
    \begin{aligned}
        \F_{\al,K}^{\eps,\eta} (p_{\T}, p_{\T}^{n}) = - \dt\left( \rho_{\al,KL} V_{KL}^{\al, \eps} - \eta (-1)^{|\al|} p_{c,KL} \right).
    \end{aligned}
\end{equation} 
\end{itemize}
\begin{rem}
The Dirichlet boundary conditions are fixed by choosing to search solution in $\R^{\T_\Dir}$. If we are looking for a solution in $\R^\T$, then we have to fix it by adding in $F^{\eps,\eta}$ the term, for $A \in \partial \M_\Dir \cup {\M^*}_\Dir$,
    \begin{equation*}
        \begin{aligned}
            &\F_{\al,A}^{\eps,\eta} (p_{\T}, p_{\T}^{n}) = 0. \\
        \end{aligned}
    \end{equation*}
\end{rem}
We approximate the porosity by its mean value on the control volume:
\begin{equation}\label{MeanPorosity}
	\phi_A = \dfrac{1}{m_A} \int_A \phi(x) \diff x .
\end{equation}
One denotes $\Z$ the continuous piecewise affine function, to force the maximum principle in the discrete equations:
\begin{equation}\label{FonctionZChap3}
    \Z(s) = \left\{
    \begin{array}{ll}
        0        &  \quad \text{if} \quad s< 0 \\
        s    & \quad \text{if} \quad s \in [0,1], \\
        1   &  \quad \text{if} \quad s>1.
    \end{array}
\right.
\end{equation}
We have the projected $\al$-phase velocity $V_{KL}^{\al,\eps}$ (resp. $V_{K^*L^*}^{\al,\eps}$) and capillary flow $p_{c,KL}$ (resp. $p_{c,K^*L^*}$) at the interface $\sig = K|L$ (resp. $\sig^* = K^*|L^*$) given by
\begin{equation}\label{V_KL_alpha}
	\begin{aligned}
 &V_{KL}^{\al,\eps} :=   M_{\al,KL}^{up,\eps} \tau_{KL} \ddKL p_\al + \sqrt{ M_{\al,KL}^{min,\eps}}  \sqrt{  M_{\al,K^*L^*}^{up,\eps}} \eta_{\D} \ddKeLe p_\al, \\
 & p_{c,KL} = \tau_{KL} \ddKL (p_{\g}- p_{\w}), \\ 
 & V_{K^*L^*}^{\al,\eps} :=   M_{\al,K^*L^*}^{up,\eps} \tau_{K^*L^*} \ddKeLe p_\al + \sqrt{ M_{\al,K^*L^*}^{min,\eps}}  \sqrt{  M_{\al,KL}^{up,\eps}} \eta_{\D} \ddKL p_\al, \\
 & p_{c,K^*L^*} = \tau_{K^*L^*} \ddKeLe (p_{\g}- p_{\w}).
	\end{aligned}
\end{equation} 
We choose the discrete mobilities (see \eqref{EpsilonMobilities}) as 
\begin{equation}
\begin{aligned}
    & M_{\al,AB}^{up,\eps} :=  \left\{
	\begin{array}{ll}
		M_{\al}^\eps(s_{\al,B}) & , \mbox{if } \ddAB p_\al  \ge 0  \\ 
		\\
		M_{\al}^\eps(s_{\al,A}) &, \mbox{otherwise} 
	\end{array}
	\right. \\
 &\text{and} \quad M_{\al,AB}^{min,\eps} :=  \min \left( M_{\al}^\eps(s_{\al,A}), M_{\al}^\eps(s_{\al,B}) \right).
\end{aligned}
\end{equation}
Moreover, we keep the relation \eqref{RelationPiUiModel} at the discrete level for all $A \in \T$ 
\begin{equation} (s_{\g,A},s_{\w,A}) = G(x_A,p_{\g,A},p_{\w,A})=  \left(\pc^{-1}(p_{\g,A} - p_{\w,A}), 1 - \pc^{-1}(p_{\g,A} - p_{\w,A}) \right). \end{equation}
One approximates the density of the $\al$-phase with an integral formula (see \cite{GhiQueSaa2020PositivityPreservingFiniteVolCompressibleTwoPhaseAni,KhaSaa2010SolutionsCompressible,SaaSaa2013StudyFullImplicitPetroleumFiniteVolCompressibleTwoPhase}). We have for all $A$, $B$ in $\overline{\M} \cup \overline{\M^*}$ 
\begin{equation} \label{rhoKL}
	\dfrac{1}{\rho_{\al,AB}} :=  \left\{
	\begin{array}{ll} \displaystyle
		\dfrac{1}{p_{\al,B} - p_{\al,A}} \int_{p_{\al,A}}^{p_{\al,B}} \dfrac{1}{\rho(z)} \diff z & \mbox{if } p_{\al,A} \neq p_{\al,B}  \\ 
		\\
		\dfrac{1}{\rho_{\al} (p_{\al,A})} & \mbox{otherwise} 
	\end{array}
	\right. .
\end{equation}

\subsubsection{Regularized PP-DDFV scheme}\label{SectionDDFVRegularizedChap3}

The objective is to demonstrate that the regularized scheme of the one studied in \cite{Crozon1} admits a solution. The continuity of all the terms makes \pointa and \pointb obvious. Let us show \pointc.

\begin{lem}[Maximum principle of the $0,\eta$-saturation]\label{maxprincEpDDFV} Let $p_{\g,\w,\T} =(p_{\g,\T},p_{\w,\T})$ be a solution to $F^{0,\eta}(p_\T,p_{\T}^n) =0$ with $\eta \ge 0$. Then, for $\al \in \{\g,\w \}$, the discrete saturation of the $\al$-phase obeys its physical bounds i.e., 
	\begin{equation} 0 \le s_{\al,A} \le 1, \quad \forall A \in \T. \end{equation}
\end{lem}

\begin{proof}[Proof]
The proof is the same as in the proof of Lemma in \cite{Crozon1}, and we will show how to handle the $\eta$-regularizing term. We take $\al = \g$, without loss of generality. We assume that for $n$ in $\llbracket 1, \Nf -1 \rrbracket$,  the property is true $(p_{\g,\T}^n, p_{\w,\T}^n)$, then we take $A \in \T$ such that $s_{\g,A} = \min_{B \in \T} s_{\g,B}$. We treat the case $A = K \in \M$, we treat the other cases likewise. One has  
	\begin{equation*}  
 \underbrace{ m_K \phi_K \left( \rho_{\g}(p_{\g,K}) \Z(s_{\g,K}) - \rho_{\g}(p_{\g,K}^{n}) s_{\g,K}^{n}   \right) (s_{\g,K})^-}_{= ACC_{K}^\g} - \dt CONV_K^\g - \eta \dt PC_K^\g = 0.
    \end{equation*}
    It is already established that  $ACC_{K}^\g \le 0$ and $CONV_K^\g \ge 0$ (see \cite{Crozon1}). Now, we look at the $\eta$ capillary pressure flow
\begin{equation*}
    \eta \dt PC_{K}^\g= \eta \dt 
\sum_{\sig = K|L \in \E_K} \rho_{\g,KL} \tau_{KL} \ddKL \pc \underbrace{{ (s_{\g,K})^-}}_{ \ge 0 }.
\end{equation*} 
Since $\pc$ is strictly increasing w.r.t $s_\g$, we have
\begin{equation*}
    \ddKL \pc (s_{\g,K})^- = \left( \pc (s_{\g,L}) - \pc (s_{\g,K}) \right) (s_{\g,K})^- \ge 0.
\end{equation*}
Then $ PC_{K}^\g \ge 0$. It implies that $s_{\g,A} \le 0$ for all $A \in \T$. 

If we reason on $\al = \w$, we have for $s_{\w,K} = \min_{B \in \T} s_{\w,B}$ 
\begin{equation*}  
 \underbrace{m_K \phi_K \left( \rho_{\w}(p_{\w,K}) \Z(s_{\w,K}) - \rho_{\w}(p_{\w,K}^{n}) s_{\w,K}^{n}   \right) (s_{\w,K})^-}_{= ACC_{K}^\w} - \dt CONV_K^\w + \eta \dt PC_K^\w = 0.
    \end{equation*}
Similarly, there holds $ACC_{K}^\w \le 0$ and $CONV_K^\w \ge 0$. We have for the $\eta$ term 
\begin{equation*}
    \eta \dt PC_{K}^\w= \eta \dt
\sum_{\sig = K|L \in \E_K} \rho_{\w,KL} \tau_{KL} \ddKL \pc \underbrace{{ (s_{\w,K})^-}}_{ \ge 0 }.
\end{equation*} 
Since $\pc$ is strictly increasing w.r.t $s_\g$, one deduces 
\begin{equation*}
    \ddKL \pc (s_{\w,K})^- = \left( \pc (1 - s_{\w,L}) - \pc (1 -s_{\w,K}) \right) (s_{\g,K})^- \le 0.
\end{equation*}
Finally $ PC_{K}^\w \le 0$. We conclude that $s_{\w,A} \ge 0$ for all $A \in \T$. Then, the item \pointc for the proposed regularized scheme is fulfilled.
\end{proof}

Let us now check the validity of the last hypothesis \pointd. This is the most demanding point in calculations. Let us set-up notations (see \eqref{NolinearFunctions} for $g_\al$)
\begin{equation*}
\begin{aligned}
        & g(p_{\g,\w,\T}) = ( g_\g(p_{\g,\T}), g_\w(p_{\w,\T})), \\
        & \F^{\eps,\eta}(p_{\g,\w,\T},p_{\g,\w,\T}^n) = ( F_\g^{\eps,\eta}(p_{\g,\w,\T},p_{\g,\w,}^n),F_\w^{\eps,\eta}(p_{\g,\w,\T},p_{\g,\w,\T}^n)) , 
\end{aligned}
\end{equation*}
where 
\begin{equation*}
\begin{aligned}
    g_\al(p_{\al,\T}) & = \left( (g_\al(p_{\al,K})_{K \in \overline{\M}_\Neu}, (g_\al(p_{\al,K^*})_{K^* \in \overline{\M^*}_\Neu} \right), \\
    \F_\al^{\eps,\eta}(p_{\g,\w,\T},p_{\g,\w,\T}^n) & = \left( (\F_{\al,K}^{\eps,\eta}(p_{\g,\w,\T},p_{\g,\w,\T}^n))_{K \in \overline{\M}_\Neu}, (\F_{\al,K^*}^{\eps,\eta}(p_{\g,\w,\T},p_{\g,\w,\T}^n))_{K^* \in \overline{\M^*}_\Neu} \right).
\end{aligned}
\end{equation*}
We compute
\begin{equation}\label{Gam1_2_3}
    \begin{aligned}
        \langle \F^{\eps,\eta}(p_{\g,\w,\T},p_{\g,\w,\T}^n) , g(p_{\g,\w,\T}) \rangle = \gamma_1 + \gamma_2 + \gamma_3,
    \end{aligned}
\end{equation}
where, treating the accumulation term as in \cite{KhaSaa2010SolutionsCompressible} and \cite{Crozon1}, using the function $\displaystyle \HH_{\al}(p_{\al})$, we have
\begin{equation}\label{CalculEstimGamma1}
    \begin{aligned}
        \gamma_1 & =  \sum_{\al \in \{\g, \w \}} \left( \sum_{K \in \M} m_K \phi_K (\rho_\al(p_{\al,K}) \Z(s_{\al,K}) - \rho_\al(p_{\al,K}^n) s_{\al,K}^n) g_\al(p_{\al,K}) \right. \\
        & \left.+ \sum_{K^* \in \overline{\M^*}_\Neu} m_{K^*} \phi_{K^*} (\rho_\al(p_{\al,{K^*}}) \Z(s_{\al,{K^*}}) - \rho_\al(p_{\al,{K^*}}^n) s_{\al,{K^*}}^n) g_\al(p_{\al,{K^*}}) \right) \\
        & \ge  \sum_{K \in \M} m_K \phi_K \sum_{\al \in \{\g, \w \}} (\HH_\al(p_{\al,K}) \Z(s_{\al,K}) - \HH_\al(p_{\al,K}^n) s_{\al,K}^n)  \\
        &  + \sum_{K \in \M} m_K \phi_K \pc(s_{\g,K})( \Z(s_{\g,K}) - s_{\g,K}^n)  \\
        & + \sum_{K^* \in \overline{\M^*}_\Neu} m_{K^*} \phi_{K^*}  \sum_{\al \in \{\g, \w \}} (\HH_\al(p_{\al,{K^*}}) \Z(s_{\al,{K^*}}) - \HH_\al(p_{\al,{K^*}}^n) s_{\al,{K^*}}^n)  . \\
        & + \sum_{K^* \in \overline{\M^*}_\Neu} m_{K^*} \phi_{K^*} \pc(s_{\g,K^*})( \Z(s_{\g,K^*}) - s_{\g,K^*}^n) .
    \end{aligned}
\end{equation}
It follows, using \eqref{IneqNorms}, that
\begin{equation}\label{EstimGamma1}
    \begin{aligned}
        \gamma_1 & \ge   \underbrace{2 \phiinf  \left( \sum_{\al \in \{ \g, \w \}} |\HH_\al(p_{\al,\T})\Z(s_{\al,\T})|_{1,\T} \right)}_{\ge 0} - \underbrace{2 \CtePhiSup \left( \sum_{\al \in \{ \g, \w \}} |\HH_\al(p_{\al,\T}^n)|_{1,\T} \right)}_{= C_n} \\
        & - 2 \phisup | p_{c,\T}|_{1,\T} \\
        & \ge  - C_n - 2 \phisup C \norme{p_{c,\T}}_{\T,\tau}= - C_n - C_{\gamma_1} \norme{p_{\g,\T} - p_{\w,\T}}_{\T,\tau}.
    \end{aligned}
\end{equation}
We point out that the constant $C_n \ge 0$. One treats $\gamma_2$ in the same fashion as in \cite{Crozon1}, appearing a constant $\nu >0$, depending on $\Laminf$, $\Lamsup$, and the fixed mesh such that
\begin{equation}\label{EstimGamma2_1}
    \begin{aligned}
        \gamma_2 & = - \dt \sum_{\al \in \{ \g,\w\}} \sum_{\D \in \DD} \rho_{\al,KL} V_{KL}^{\al,\eps} \ddKL g_\al(p_\al) + \rho_{\al,K^*L^*} V_{K^*L^*}^{\al,\eps} \ddKeLe g_\al(p_\al) \\
        & \ge \dt \nu \sum_{\al \in \{ \g,\w\}} \sum_{\D \in \DD} M_{\al,KL}^{up,\eps} \tau_{KL} (\ddKL p_{\al})^2 + M_{\al,K^*L^*}^{up,\eps} \tau_{K^*L^*} (\ddKeLe p_{\al})^2 \\
        & \ge \dt \nu \eps \sum_{\al \in \{ \g,\w\}} \sum_{\D \in \DD} \tau_{KL} (\ddKL p_{\al})^2 + \tau_{K^*L^*} (\ddKeLe p_{\al})^2 = \dt \nu \eps \sum_{\al \in \{ \g,\w\}} \norme{p_{\al,\T}}_{\T,\tau}^2.
    \end{aligned}
\end{equation}
We now make use of the following Lemma (coming from the works \cite{GhiQueSaa2020PositivityPreservingFiniteVolCompressibleTwoPhaseAni,EymHerMic2003MathematicalStudyPetroleumEngineeringScheme}).
\begin{lem}\label{lem1}
	For every $A,B$ in $\overline{\M} \cup \overline{\M^*}$, there holds: 
\begin{equation*} \label{IneqPressPhaseGlobCap} 
m_0  \left( \left( \ddAB p \right)^2 +  \left( \ddAB \xi\right)^2 \right) \le M_{\g,AB}^{up} \left( \ddAB p_\g \right)^2 +  M_{\w,AB}^{up}  \left(\ddAB p_\w \right)^2 . 
 \end{equation*}
\end{lem}
Then, we also deduce
\begin{equation}\label{EstimGamma2_2}
    \begin{aligned}
        \gamma_2  \ge \dt \nu & \left(  \sum_{\D \in \DD} \tau_{KL} (\ddKL p)^2  + \tau_{K^*L^*} (\ddKeLe p)^2 \right. \\
        & \left. + \sum_{\D \in \DD} \tau_{KL} (\ddKL \xi)^2 + \tau_{K^*L^*} (\ddKeLe \xi)^2 \right) \\
         =\dt \nu & \left( \norme{p_\T}_{\T,\tau}^2 + \norme{\xi_\T}_{\T,\tau}^2 \right) . \\
    \end{aligned}
\end{equation}
Finally, the discrete integration-by-parts gives us
\begin{equation}\label{EstimGamma3}
    \begin{aligned}
        \gamma_3 & = - \eta \dt \sum_{\al \in \{ \g, \w \}} (-1)^{|\al|}  \sum_{\D \in \DD} \rho_{\al,KL} p_{c,KL} \ddKL g_\al(p_\al) + \rho_{\al,K^*L^*} p_{c,K^*L^*} \ddKeLe g_\al(p_\al) \\
        & = - \eta \dt \sum_{\D \in \DD} p_{c,KL} \ddKL (p_\g - p_\w) + p_{c,K^*L^*} \ddKeLe (p_\g - p_\w) \\
        & = \eta \dt \sum_{\D \in \DD} \tau_{KL} (\ddKL \pc)^2 + \tau_{K^*L^*} (\ddKeLe \pc)^2 = \eta \dt \norme{p_{\g,\T} - p_{\w,\T}}_{\T,\tau}^2 .\\
    \end{aligned}
\end{equation}
Finally, we obtain using \eqref{Gam1_2_3}, \eqref{EstimGamma1}, \eqref{EstimGamma2_1}, \eqref{EstimGamma2_2}, \eqref{EstimGamma3} and the fact that since $p_{\T}$ is solution to $\F^{\eps,\eta}(p_{\g,\w,\T},p_{\g,\w,\T}^n) =0$, $\gamma_1 + \gamma_2 + \gamma_3 =0$, we have
\begin{equation} \label{EquationExistence}
    \begin{aligned}
        0 & = \langle \F^{\eps,\eta}(p_{\g,\w,\T},p_{\T}^n) , g(p_{\g,\w,\T}) \rangle \\
        &\ge - C_{\gamma_1} \norme{p_{\g,\T} - p_{\w,\T}}_{\T,\tau}  -C_n  +  \dt \nu \left( \norme{p_\T}_{\T,\tau}^2 + \norme{\xi_\T}_{\T,\tau}^2 \right) \\
        & + \dt \nu \eps \left( \sum_{\al \in \{ \g,\w\}} \norme{p_{\al,\T}}_{\T,\tau}^2 \right) + \eta \dt \norme{p_{\g,\T} - p_{\w,\T}}_{\T,\tau}^2.
    \end{aligned}
\end{equation}
As a result, $ C_{\gamma_1} \norme{p_{\g,\T} - p_{\w,\T}}_{\T,\tau} \le \dfrac{\eta \dt}{2} \norme{p_{\g,\T} - p_{\w,\T}}_{\T,\tau}^2 + \dfrac{C_{\gamma_1}^2}{2\eta \dt} $. Following
\begin{equation*}
    \begin{aligned}
        C_n + \dfrac{C_{\gamma_1}^2}{2\eta \dt} & \ge  \dt \nu \left( \norme{p_\T}_{\T,\tau}^2 + \norme{\xi_\T}_{\T,\tau}^2 \right)  + \dt \nu \eps \left( \sum_{\al \in \{ \g,\w\}} \norme{p_{\al,\T}}_{\T,\tau}^2 \right) + \\
        & \dfrac{\eta \dt}{2} \norme{p_{\g,\T} - p_{\w,\T}}_{\T,\tau}^2.
    \end{aligned}
\end{equation*}
The result is \pointd. 

For the last point, since \pointc holds true, taking over the computation \eqref{CalculEstimGamma1}-\eqref{EstimGamma1}, one obtains
\begin{equation*}
    \begin{aligned}
        \gamma_1 \ge  - C_n - 2 m_\Om \phisup \norme{p_{c}}_{[0,1],\infty}.
    \end{aligned}
\end{equation*}
Finally \pointe is satisfied 
\begin{equation*}
    \begin{aligned}
        C_n + 2 m_\Om \phisup \norme{p_{c}}_{[0,1],\infty} \ge  \dt \nu \left( \norme{p_\T}_{\T,\tau}^2 + \norme{\xi_\T}_{\T,\tau}^2 \right)  .
    \end{aligned}
\end{equation*}

\subsubsection{Existence}

Now, it only remains to prove the existence of a solution to the regularized PP-DDFV scheme for every $\eps > 0$ and $\eta >0$. We will use the important fixed point result \cite{Evans2022PDE}
\begin{lem}[Zeros of vector fields \cite{Evans2022PDE}]\label{lemzeros}
	Assume the continuous function $v: \R^n \rightarrow \R^n$ satisfies $ v(x) \cdot x \ge 0$, if $|x|=r$, for some $r>0$. Then there exists a point $x \in B(0,r)$ such that $v(x)=0$.
\end{lem}

For proof of the existence of an approximate solution of a similar model one can refer to \cite{GhiQueSaa2019PositiveControlVolFiniteElementDegenerateCompressibleAnistropic,SaaSaa2014NumericalAnalysisNonEquilibrium2phase2componentCompressible,BilalSaadThesis2011}.

	We write $ n = Card({ \overline{\M}}_\Neu)$, $m= Card({ \overline{\M^*}_\Neu})$. One defines the $C^1$-diffeopmorphism, verifying $\Phi(0) = 0$, 
 \begin{equation*}
     \fonction{\Phi}{\R^n \times \R^m \times \R^n \times \R^m}{\R^n \times \R^m \times \R^n \times \R^m}{(a, a^*, b, b^*)}{(g_\g(a),g_\g(a^*),g_\w(b),g_\w(b^*))} .
 \end{equation*}
 We write $\Phi^{-1}(u,u^*,w,w^*) = (p_\g,p_\g^*,p_\w,p_\w^*)$, then we apply Lemma \ref{lemzeros} to \\ $v(.) = \F^{\eps,\eta} (\Phi^{-1}(.), p_{\g,\w,\T}^n)$. We choose the norm on $\R^n \times \R^m \times \R^n \times \R^m$, given by 
 \begin{equation*}
     \norme{(u,u^*,w,w^*)}^2 = \norme{g_\g^{-1}(u,u^*)}_{\T,\tau}^2 + \norme{g_\w^{-1}(w,w^*)}_{\T,\tau}^2. 
 \end{equation*}
 Thanks, to \eqref{EquationExistence}, and using that $\sqrt{2(a^2 + b^2)} \ge a + b$ for all $a$, $b\ge 0$, we write
 \begin{equation*}
     \begin{aligned}
         \langle v(u,u^*,w,w^*), (u,u^*,w,w^*) \rangle \ge &  - C_{\gamma_1} \sqrt{2} \norme{(u,u^*,w,w^*)} - C_n \\ 
         & + \dt \nu \eps \norme{(u,u^*,w,w^*)}^2.
     \end{aligned}
 \end{equation*}
 Taking $r \ge 0$, such that $ \dt \nu \eps r^2 - C_{\gamma_1} \sqrt{2} r - C_n \ge 0$, one applies Lemma \ref{lemzeros}, then there exists $(u,u^*,w,w^*)$ in $B_{\norme{.}}(0,r)$ such that $v(u,u^*,w,w^*) = 0$, meaning that $\Phi^{-1}(u,u^*,w,w^*)$ verifies the equation of the regularized PP-DDFV scheme. Then it admits a solution.

\subsection{CVFE scheme for compressible two-phase flow}\label{CVFEApplication}

In this section, we propose a regularization for the Control Volume Finite Element (CVFE) scheme introduced in \cite{GhiQueSaa2020PositivityPreservingFiniteVolCompressibleTwoPhaseAni}. Then, we show the existence of a solution to this regularized version of the scheme, implying a solution for the original one. This subsection follows exactly the same structure as in subsection \ref{BPDDFVApplication}.

\subsubsection{CVFE settings}

This method is a vertex centered method. The mesh $\T$  is a conforming simplicial partition of the domain $\Om$ ( in the sense of the finite element see \cite{Ern2004TheoryFE}). 

In 2D it consists of  a triangular mesh, such that for two distinct elements $T$, $T'$, $T \cap T'$ can be either a common vertex, and edge or the emptyset. One denotes   $x_T$  barcyenter of $T$. The set of the vertices of the mesh is written $\VV$. For $x_K$ a vertex of $T$, we write $\VV_{KT}$ the  vertices of $T$ except $x_K$. The vertices of the mesh will be the  degrees of freedom.  Moreover, we build a dual mesh around these vertices of $\T$. For a node $x_K$ of $\VV$ one associates a unique control volume $K$.  Let $\T_K$ be the set of triangles sharing $x_K$ as vertex. Then, the fraction in the triangle $T$ of cell $K$ is given by the polygon $A_K^T \subset T$, whose vertices are $x_K$, $x_T$ and the two midpoint of the segments $[x_K,x_L]$ for $L \in \VV_{KT}$. Therefore,  the control volume associated to $x_K$ is defined by $\overline{A_K} = \cup_{T \in \T_K} \overline{A_K^T}$. We call $\A$ this dual mesh centered on the nodes of the initial mesh. Due to  Dirichlet boundary conditions, we are led to  distinguish $\VVD$ the vertices located on $\GamD$, from the others $\VVDc = \VV \backslash \VVDc$. For more details on the CVFE setting one can refer to \cite{GhiQueSaa2020PositivityPreservingFiniteVolCompressibleTwoPhaseAni,CanIbrSaa2017PositiveCVFEDegenerateKellerSegel,CanGui2016ConvergenceNonlinearEntropyCVFEAniDegeneratePrabolicEqu}.

\subsubsection{Discrete operators and functions}

Let  $V_\T$ be  the $\mathbb{P}_1$ finite elements function space on the mesh $\T$. It is composed of the continuous and piecewise affine functions per elements: 
\begin{equation*}
    V_\T = \left\{ f \in C^0(\overline{\Om}), \quad f_{|T} \in \mathbb{P}_1, \quad \forall T \in \T \right\}.
\end{equation*}
This function space is in $H^1(\Om)$. The shape function basis $(\varphi_A)_{A \in \VV}$ are the elements of $V_\T$ such that $\varphi_A(x_B) =1$ if $A=B$ and $\varphi_A(x_B)= 0$ otherwise. One has the following the relations  $\sum_{K \in \VV} \varphi_K = 1$ and $\sum_{K \in \VV} \nabla \varphi_K = 0$. One decomposes each element $f$ of $ V_\T$ and its gradient likewise
\begin{equation*}
    f = \sum_{K \in \VV} f_K \varphi_K, \quad \nabla f = \sum_{K \in \VV} f_K \nabla \varphi_K. 
\end{equation*}
Considering the space associated to the Dirichlet boundary conditions
\begin{equation*}
    V_\T^0 = \left\{ f \in \VV, \quad f(x_K) = 0 \in , \quad \forall K \in \VVD \right\}.
\end{equation*}
We have the natural semi-norm on $V_\T$
\begin{equation*}
    \norme{f}_{V_\T}^2 = \int_\Om \norme{\nabla f}_2^2 \diff x.
\end{equation*}
It became a norm on $V_\T^0$ because of the discrete Poincaré inequality \cite{Ern2004TheoryFE}. We also use the norm, defined for $f_\T$ in $V_\T$ by
\begin{equation*}
    |f_\T|_{1,\T} = \sum_{K \in \VV} m_{A_K} |f_K|.
\end{equation*}
Then, the discrete Poincaré's inequality holds true, i.e. there exist a constant $C>0$, such that, for all $f_\T $ in $V_\T$:
\begin{equation}\label{IneqNormsCVFE}
    |f_\T|_{1,\T} \le C \norme{f_\T}_{V_\T}.
\end{equation}

Like for the BP-DDFV scheme, we split the time interval into subintervals $[t^n,t^{n+1}[$ such that $0=t_0 < t^1< ...< t^\Nf= \tf$, and take $\dt = t^{n+1}-t^n$.

\subsubsection{Presentation of the CVFE scheme}

We introduce the implicit positivity preserving CVFE method (see \cite{GhiQueSaa2020PositivityPreservingFiniteVolCompressibleTwoPhaseAni}) \eqref{DebutSchemeCVFE}-\eqref{RelationSiPiCVFE}.  The old solution  $p_{\alpha,\w,\T}^n$ belongs to $V_\T^0$, and the deduced saturation is verifying the discrete maximum principle. We keep the notations $|g| = 0$, $|\w| = 1$. We are looking for $p_{\alpha ,\w,\T}$ in $V_\T^0$ solution to the regularized scheme, given by $\F_\alpha^{\eps,\eta}$ as follow, with $\eps \ge 0$, $\eta \ge 0$, for $K \in \VVDc$,
    \begin{equation}\label{DebutSchemeCVFE}
\begin{aligned}
    \F_{\al,K}^{\eps,\eta} (p_{\g,\w,\T}, p_{\g,\w,\T}^{n}) = & m_{A_K} \phi_K \left( \rho_{\al}(p_{\al,K}) \Z(s_{\al,K}) - \rho_{\al}(p_{\al,K}^{n}) s_{\al,K}^{n} \right) \\
    & -  \dt  \sum_{T \in \T_K} \sum_{L \in \VV_{KT}} \rho_{\al, KL} M_\al^{\eps}(s_{\al,KL}) \Lam_{KL}^T \ddKL p_\al \\
    & - \dt \eta (-1)^{|\al|} \sum_{T \in \T_K} \sum_{L \in \VV_{KT}^+} \rho_{\al, KL} |\Lam_{KL}^T| \ddKL \pc.
\end{aligned}
\end{equation}

\begin{rem}
The Dirichlet boundary conditions are fixed by choosing to search solutions in $V_\T^0$. If we are looking for a solution in $V_\T$, then we have to add more equations in $\F^{\eps,\eta}$, one for every $K \in \partial \VVD$, such that
    \begin{equation}
        \begin{aligned}
            &\F_{\al,K}^{\eps,\eta} (p_{\g,\w,\T}, p_{\g,\w,\T}^{n}) = 0. \\
        \end{aligned}
    \end{equation}
\end{rem}
The transmissibility or stiffness coefficients between two neighboring control volumes $A_K$ and $A_L$ in the element $T$ are given by 
\begin{equation}
    \begin{aligned}
        \Lam_{KL}^T = - \int_T \Lam(x) \nabla \varphi_K \cdot \nabla \varphi_L \diff x = \Lam_{LK}^T.
    \end{aligned}
\end{equation}
One sets $\VV_{KT}^+$ the vertices of $T$ except $K$ such that the stiffness coefficient at the interface $\sig_{KL}^T$, $\Lam_{KL}^T$ is non-negative.

We approximate the porosity by its mean value on the control volume $A_K$ as in \eqref{MeanPorosity}. $\Z$ is defined in \eqref{FonctionZChap3}. The approximate density is still given by \eqref{rhoKL}. The saturation of the $\al$-phase on the interface $\sig_{KL}^T$ is chosen in a nonstandard way when the transmissibilities are non-negative 
\begin{equation}
	s_{\al,KL} :=  \left\{
	\begin{array}{ll}
		\left\{ \begin{array}{ll}
		s_{\al,L}  & , \quad \mbox{if } \quad \ddKL p_\al  \ge 0  \\ 
		s_{\al,K}  & , \quad \mbox{if} \quad  \ddKL p_\al < 0
	\end{array} \right. &, \quad \mbox{if} \quad \Lam_{KL}^T \ge 0  \\ 
		\quad \min_{J \in \VV_T} (s_{\al,J}) &, \quad \mbox{if } \quad \Lam_{KL}^T < 0
	\end{array}. \right. 
\end{equation}
Moreover, we keep the relation \eqref{RelationPiUiModel} at the discrete level for all $K \in \VV$ 
\begin{equation}\label{RelationSiPiCVFE} (s_{\g,K},s_{\w,K}) = G(x_K,p_{\g,K},p_{\w,K})=  \left(\pc^{-1}(p_{\g,K} - p_{\w,K}), 1 - \pc^{-1}(p_{\g,K} - p_{\w,K}) \right). \end{equation}

\subsubsection{Regularized CVFE scheme}\label{CVFERegularizedChap3}

Next, we show that \eqref{DebutSchemeCVFE}-\eqref{RelationSiPiCVFE} is a regularized scheme of the one studied in \cite{GhiQueSaa2020PositivityPreservingFiniteVolCompressibleTwoPhaseAni}. To begin with, \pointa and \pointb are obviously true by continuity of all the terms. Let us now  show \pointc

\begin{lem}[Maximum principle of the $0,\eta$-saturation]\label{maxprincEpCVFE} Let $p_{\alpha,\w,\T} =(p_{\g,\T},p_{\w,\T})$ be a solution to $\F_\alpha^{0,\eta}(p_\T,p_{\T}^n) =0$ with $\eta \ge 0$. Then, for $\al \in \{\g,\w \}$, the discrete saturation of the $\al$-phase obeys its physical ranges i.e. 
	\begin{equation} 0 \le s_{\al,A} \le 1, \quad \forall A \in \T. \end{equation}
\end{lem}

\begin{proof}[Proof]
The proof is the same as in the proof of Lemma 4.1 in \cite{GhiQueSaa2020PositivityPreservingFiniteVolCompressibleTwoPhaseAni}, we just look at the $\eta$-regularizing term. We take $\al = \g$, without loss of generality. We assume that for $n$ in $\llbracket 1, \Nf -1 \rrbracket$, the property is true for $(p_{\g,\T}^n, p_{\w,\T}^n)$. Then, we take $K \in \VVDc$ such that $s_{\g,K} = \min_{L \in \VVDc} s_{\g,L}$. We treat the case $ K \in \VVDc$, since we directly have the property for $K \in \VVD$ thanks to the Dirichlet boundary conditions and \eqref{RelationSiPiCVFE}. One has
\begin{equation*}  
 \underbrace{ m_{A_K} \phi_K \left( \rho_{\g}(p_{\g,K}) \Z(s_{\g,K}) - \rho_{\g}(p_{\g,K}^{n}) s_{\g,K}^{n}   \right) (s_{\g,K})^-}_{= ACC_{K}^\g} - \dt CONV_K^\g - \eta \dt PC_K^\g = 0.
    \end{equation*}
It is already demonstrated that  $ACC_{K}^\g \le 0$ and $CONV_K^\g \ge 0$ (see \cite{GhiQueSaa2020PositivityPreservingFiniteVolCompressibleTwoPhaseAni}). Now, we look a the $\eta$-capillary pressure flow
\begin{equation*}
    \eta \dt PC_{K}^\g= \dt \eta \sum_{T \in \T_K} \sum_{L \in \VV_{KT}^+} \rho_{\al, KL} |\Lam_{KL}^T| \ddKL \pc \underbrace{{ (s_{\g,K})^-}}_{ \ge 0 }.
\end{equation*} 
Because $\pc$ is strictly increasing w.r.t $s_\g$, we deduce
\begin{equation*}
    \ddKL \pc (s_{\g,K})^- = \left( \pc (s_{\g,L}) - \pc (s_{\g,K}) \right) (s_{\g,K})^- \ge 0.
\end{equation*}
Then, $ PC_{K}^\g \ge 0$. It implies that $s_{\g,K} \le 0$ for all $K \in \VV$.  Furthermore, we proceed similarly in the case $\al = \w$, and prove as in the proof of Lemma \ref{maxprincEpDDFV}, $ \min_{B \in \T} s_{\w,B} \ge 0$. Because of the relation between the saturations, we are able to conclude. Thus, one has \pointc for the proposed regularized CVFE scheme.
\end{proof}

It remains to prove \pointd. Using the nonlinear function \eqref{NolinearFunctions}, we set 
\begin{equation*}
\begin{aligned}
        g(p_{\g,\w,\T}) & = \left( (g_\g(p_{\g,\T}))_{K \in \VVDc}, (g_\w(p_{\w,\T}))_{K \in \VVDc} \right), \\ 
        \F^{\eps,\eta}(p_{\g,\w,\T},p_{\g,\w,\T}^n) & = \left( (\F_\g^{\eps,\eta}(p_{\g,\w,\T},p_{\g,\w,\T}^n))_{K \in \VVDc}, (\F_\w^{\eps,\eta}(p_{\g,\w,\T},p_{\g,\w,\T}^n))_{K \in \VVDc} \right).  
\end{aligned}
\end{equation*}
We compute
\begin{equation}\label{Gam1_2_3CVFE}
    \begin{aligned}
        \langle \F^{\eps,\eta}(p_{\g,\w,\T},p_{\g,\w,\T}^n) , g(p_{\g,\w,\T}) \rangle = \gamma_1 + \gamma_2 + \gamma_3,
    \end{aligned}
\end{equation}
where, treating the accumulation term $\gamma_1$ as in \cite{KhaSaa2010SolutionsCompressible,GhiQueSaa2020PositivityPreservingFiniteVolCompressibleTwoPhaseAni}, using $\HH_{\al}$, we obtain as in \eqref{EstimGamma1}
\begin{equation*}\label{CalculEstimGamma1CVFE}
    \begin{aligned}
        \gamma_1 = & \sum_{\al \in \{\g, \w \}} \left( \sum_{K \in \VVDc} m_{A_K} \phi_K (\rho_\al(p_{\al,K}) \Z(s_{\al,K}) - \rho_\al(p_{\al,K}^n) s_{\al,K}^n) g_\al(p_{\al,K}) \right) \\
        \ge &  \sum_{K \in \VVDc} m_{A_K} \phi_K  \sum_{\al \in \{\g, \w \}} (\HH_\al(p_{\al,K}) \Z(s_{\al,K}) - \HH_\al(p_{\al,K}^n) s_{\al,K}^n) \\
        & + \sum_{K \in \VVDc} m_{A_K} \phi_K \pc(s_{\g,K})( \Z(s_{\g,K}) - s_{\g,K}^n).
    \end{aligned}
\end{equation*}
We still estimate $\gamma_1$ as in the previous subsection, using \eqref{IneqNormsCVFE},
\begin{equation}\label{EstimGamma1CVFE}
    \begin{aligned}
        \gamma_1 \ge & \underbrace{\phiinf \left( \sum_{\al \in \{ \g, \w \}} |\HH_\al(p_{\al,\T}) \Z(s_{\al,\T})|_{1,\T} \right)}_{\ge 0} - \underbrace{ \CtePhiSup \left( \sum_{\al \in \{ \g, \w \}} |\HH_\al(p_{\al,\T}^n)|_{1,\T} \right)}_{= C_n} \\
        & - \phisup | p_{c,\T} |_{1,\T} \\
        \ge & - C_n - \phisup C \norme{p_{c,\T}}_{V_\T} =  - C_n - C_{\gamma_1}  \norme{p_{\g,\T} - p_{\w,\T}}_{V_\T}.
    \end{aligned}
\end{equation}
We point out that the constant $C_n$ is positive. One can deal with   $\gamma_2$, in the same fashion as in Proposition 4.1 (see \cite{GhiQueSaa2020PositivityPreservingFiniteVolCompressibleTwoPhaseAni})
\begin{equation}\label{EstimGamma2_1CVFE}
    \begin{aligned}
        \gamma_2 & = - \dt \sum_{\al \in \{\g,\w\}} \sum_{K \in \VVDc} \sum_{T \in \T_K} \sum_{L \in \VV_{KT}} \rho_{\al, KL} M_\al^{\eps}(s_{\al,KL}) \Lam_{KL}^T \ddKL p_\al g_\al(p_{\al,K})  \\
        & = - \dt \sum_{\al \in \{\g,\w\}} \sum_{K \in \VV} \sum_{T \in \T_K} \sum_{L \in \VV_{KT}} \rho_{\al, KL} M_\al^{\eps}(s_{\al,KL}) \Lam_{KL}^T \ddKL p_\al g_\al(p_{\al,K})  \\
        & = \dt \sum_{\al \in \{\g,\w\}} \sum_{T \in \T} \sum_{\sig_{KL}^T \in \E_T} \rho_{\al, KL} M_\al^{\eps}(s_{\al,KL}) \Lam_{KL}^T \ddKL p_\al \ddKL g_\al(p_{\al}) \\
        & = \dt \sum_{\al \in \{\g,\w\}} \sum_{T \in \T} \sum_{\sig_{KL}^T \in \E_T} M_\al^{\eps}(s_{\al,KL}) \Lam_{KL}^T (\ddKL p_\al)^2 \ge \dt \eps \Laminf \sum_{\al \in \{\g,\w\}} \norme{p_{\al,\T}}_{V_\T}^2. \\
    \end{aligned}
\end{equation}
A constant $\nu >0$ appears \cite{GhiQueSaa2020PositivityPreservingFiniteVolCompressibleTwoPhaseAni}, depending on the fixed mesh and on the permeability bounds $\Laminf$, $\Lamsup$. But, we also have thanks to Lemma 4.2 and 4.3 (see \cite{GhiQueSaa2020PositivityPreservingFiniteVolCompressibleTwoPhaseAni}), where $\nu$ is the constant given in Lemma 4.3, that 
\begin{equation}\label{EstimGamma2_2CVFE}
    \begin{aligned}
        \gamma_2 & \ge \dt m_0 \nu \sum_{T \in \T} \sum_{\sig_{KL}^T \in \E_T} \Lam_{KL}^T \left( (\ddKL p)^2 + (\ddKL \xi)^2 \right) \\
        & \ge \dt m_0 \nu \Laminf \left( \norme{p_{\T}}_{V_\T}^2 + \norme{\xi_{\T}}_{V_\T}^2 \right). \\
    \end{aligned}
\end{equation}
Finally, some  computations give us
\begin{equation}\label{EstimGamma3CVFE}
    \begin{aligned}
        \gamma_3 & = - \eta \dt \sum_{\al \in \{ \g,\w\}} (-1)^{|\al|} \sum_{K \in \VVDc} \sum_{T \in \T_K} \sum_{L \in \VV_{KT}^+} \rho_{\al, KL} |\Lam_{KL}^T| \ddKL \pc g_\al (p_{\al,K}) \\
         & = \eta \dt \sum_{\al \in \{ \g,\w\}} (-1)^{|\al|} \sum_{T \in \T} \sum_{ \sig_{KL}^T \in \E_T^+} \rho_{\al, KL} |\Lam_{KL}^T| \ddKL \pc  \ddKL g_\al (p_{\al}) \\
         & = \eta \dt \sum_{T \in \T} \sum_{ \sig_{KL}^T \in \E_T^+} |\Lam_{KL}^T| \ddKL \pc \left( \ddKL p_{\g} - \ddKL p_{\w}\right) \\
         & = \eta \dt \sum_{T \in \T} \sum_{ \sig_{KL}^T \in \E_T^+} |\Lam_{KL}^T| (\ddKL \pc)^2 \\
         & \ge \eta \dt \sum_{T \in \T} \sum_{ \sig_{KL}^T \in \E_T} \Lam_{KL}^T (\ddKL \pc)^2 \ge \eta \dt \Laminf \norme{p_{c,\T}}_{V_\T}^2  .\\
    \end{aligned}
\end{equation}
Thus, using \eqref{Gam1_2_3CVFE}, \eqref{EstimGamma1CVFE}, \eqref{EstimGamma2_1CVFE}, \eqref{EstimGamma2_2CVFE}, \eqref{EstimGamma3CVFE} and the fact that since $p_{\g,\w,\T}$ is solution to $\F^{\eps,\eta}(p_{\g,\w,\T},p_{\g,\w,\T}^n) =0$, $\gamma_1 + \gamma_2 + \gamma_3 =0$, we obtain
\begin{equation} \label{EquationExistenceCVFE}
    \begin{aligned}
        0 = & \langle \F^{\eps,\eta}(p_{\g,\w,\T},p_{\g,\w,\T}^n) , g(p_{\g,\w,\T}) \rangle \\ 
        \ge & - C_{\gamma_1}  \norme{p_{\g,\T} - p_{\w,\T}}_{V_\T} - C_n +  \dt \nu \Laminf \left( \norme{p_\T}_{V_\T}^2 + \norme{\xi_\T}_{V_\T}^2 \right) \\
        & + \dt \nu \eps \Laminf \sum_{\al \in \{ \g,\w\}} \norme{p_{\al,\T}}_{V_\T}^2 + \eta \dt \Laminf \norme{p_{\g,\T} - p_{\w,\T}}_{V_\T}^2 .
    \end{aligned}
\end{equation}
One claims $ C_{\gamma_1} \norme{p_{\g,\T} - p_{\w,\T}}_{V_\T} \le \dfrac{\eta \dt \Laminf}{2} \norme{p_{\g,\T} - p_{\w,\T}}_{V_\T}^2 + \dfrac{C_{\gamma_1}^2}{2\eta \dt \Laminf} $. Consequently, there holds 
\begin{equation*}
    \begin{aligned}
        C_n + \dfrac{C_{\gamma_1}^2}{2\eta \dt \Laminf} \ge & \dt \nu \Laminf \left( \norme{p_\T}_{V_\T}^2 + \norme{\xi_\T}_{V_\T}^2 \right)  + \dt \nu \eps \Laminf \left( \sum_{\al \in \{ \g,\w\}} \norme{p_{\al,\T}}_{V_\T}^2 \right) \\
        & + \dfrac{\eta \dt \Laminf}{2} \norme{p_{\g,\T} - p_{\w,\T}}_{V_\T}^2.
    \end{aligned}
\end{equation*}
As a result, \pointd is satisfied. 

For the last point, according to the item\pointc, and taking over the computation \eqref{CalculEstimGamma1CVFE}-\eqref{EstimGamma1CVFE}, one obtains
\begin{equation*}
    \begin{aligned}
        \gamma_1 \ge  - C_n - m_\Om \phisup \norme{p_{c}}_{[0,1],\infty}.
    \end{aligned}
\end{equation*}
It follows \pointe
\begin{equation*}
    \begin{aligned}
        C_n +  m_\Om \phisup \norme{p_{c}}_{[0,1],\infty} \ge  \dt \nu \Laminf \left( \norme{p_\T}_{V_\T}^2 + \norme{\xi_\T}_{V_\T}^2 \right)  .
    \end{aligned}
\end{equation*}

\subsubsection{Existence}

Now, it only remains to prove the existence of a solution to the regularized CVFE scheme, for every $\eps>0$ and $\eta >0$. We write $ n = Card( \VVDc)$. One defines the $C^1$-diffeopmorphism, verifying $\Phi(0) = 0$, 
 \begin{equation*}
     \fonction{\Phi}{\R^n \times  \R^n }{\R^n \times \R^n}{(a, b)}{(g_\g(a),g_\w(b))} .
 \end{equation*}
 We write $\Phi^{-1}(u,w) = (p_\g,p_\w)$, then we apply Lemma \ref{lemzeros} to \\
 $v(.) = \F^{\eps,\eta} (\Phi^{-1}(.), p_{\g,\w,\T}^n)$. We choose the norm on $\R^n \times \R^n$ given by 
 \begin{equation*}
     \norme{(u,w)}^2 = \norme{g_\g^{-1}(u)}_{V_\T}^2 + \norme{g_\w^{-1}(w)}_{V_\T}^2. 
 \end{equation*}
 Thanks to \eqref{EquationExistenceCVFE}, and using that $\sqrt{2(a^2 + b^2)} \ge a + b$ for all $a$, $b\ge 0$, lead to 
 \begin{equation*}
     \begin{aligned}
         \langle v(u,w), (u,w) \rangle \ge  - C_{\gamma_1} \sqrt{2} \norme{(u,w)} - C_n + \dt \nu \eps \norme{(u,w)}^2.
     \end{aligned}
 \end{equation*}
 Taking the radius $r \ge 0$, such that $ \dt \nu \eps \Laminf r^2 - C_{\gamma_1} \sqrt{2} r - C_n \ge 0$, one applies Lemma \ref{lemzeros}, then there exists $(u,w)$ in $B_{\norme{.}}(0,r)$ such that $v(u,w) = 0$, meaning that $\Phi^{-1}(u,w) = (p_\g,p_\w)$ is a solution to the regularized CVFE scheme, which finishes the proof.

\section{Conclusions}

In this paper, we propose a framework to prove rigorously the existence of solutions to some numerical schemes sharing some structural properties of stability. We try to catch a large variety of continuous models, which encompass the two-phase Darcy flow in porous media model. We aim for Euler implicit time-discretization, but it can be applied to other types of schemes. A few key assumptions need to be fulfilled to use the result: a formal relationship between the unknowns identical to the one of the continuous model, a maximum principle, and lastly, energy estimates. The key idea is to build regularized versions of the considered schemes so that proving the existence of solutions to the schemes in question is simpler. Then, it will imply a solution to the original numerical scheme first studied.

We illustrate the use of this tool in the case of the two-phase Darcy flow. It enables us to handle the degeneracy, which was an issue to show the existence. First, we treat a Positivity-Preserving DDFV scheme. In a second time, the method is applied to a CVFE scheme. The idea of both regularizations is, on the one hand, to $\epsilon$-perturb mobilities for removing the degeneracy impact, and on the other hand, to add a capillary pressure flow with positive coefficients. These applications illustrate the strength of the proposed approach and demonstrate its broad applicability and potential for generalization across diverse contexts.


\bigskip
\noindent\textbf{Acknowledgment}: the authors would like to thank the FMPL and the Ecole Centrale Nantes for supporting this work. 



\bibliographystyle{abbrv}
\bibliography{biblio}

\end{document}